\documentclass[a4paper]{article}
\usepackage[english]{babel}
\usepackage[utf8x]{inputenc}
\usepackage[T1]{fontenc}

\usepackage[a4paper,top=3cm,bottom=2cm,left=3cm,right=3cm,marginparwidth=1.75cm]{geometry}

\usepackage{graphicx, amsmath, amssymb, amsthm, amscd}
\usepackage{color}
\usepackage{epsf}
\usepackage{enumerate}

\usepackage{amsmath}
\usepackage{graphicx}
\usepackage[colorinlistoftodos]{todonotes}
\usepackage[colorlinks=true, allcolors=blue]{hyperref}

\newtheorem{theorem}{Theorem}[section]

\newtheorem{claim}[theorem]{Claim}

\newtheorem{remark}[theorem]{Remark}

\newcommand{\cl}{\operatorname{cl}}
\title{Linked orbits of homeomorphisms of the plane and Gambaudo-Kolev Theorem}
\begin{document}
\maketitle
\begin{center}J. P. Boro\'nski\footnote{Faculty of Mathematics and Computer Science, Jagiellonian University in Krak\'ow, ul. Łojasiewicza 6, 30-348 Kraków, Poland -- and -- National Supercomputing Centre IT4Innovations, Division of the University of Ostrava, Institute for Research and Applications of Fuzzy Modeling, 30. dubna 22, 70103 Ostrava, Czech Republic, e-mail: jan.boronski@osu.cz}
\end{center}
\begin{center}
{\small {\it Dedicated to the memory of my dad, Wojtek Boroński (1947-2021), \\
who taught me how to solve linear equations in one variable.}}
\end{center}
\begin{abstract}
Let $h : \mathbb{R}^2 \to  \mathbb{R}^2$ be an orientation preserving homeomorphism of the plane. For any bounded orbit $\mathcal{O}(x)=\{h^n(x):n\in\mathbb{Z}\}$ there exists a fixed point $p\in\mathbb{R}^2$ of $h$ linked to $\mathcal{O}(x)$ in the sense of Gambaudo: one cannot find a Jordan curve $C\subseteq\mathbb{R}^2$ around  $\mathcal{O}(x)$, separating it from $p$, that is isotopic to $h(C)$ in $\mathbb{R}^2\setminus\left(\mathcal{O}(x)\cup\{p\}\right)$. 
\end{abstract}

\section{Introduction}
Given a set $D$ by $\cl(D),\operatorname{int}(D)$, and $\partial D$ we shall denote the closure, interior and boundary of $D$ respectively. $\mathbb{S}^n$ denotes the $n$-dimensional sphere, and $\mathbb{R}^2$ denotes the plane. Given an orientation preserving homeomorphism $h:\mathbb{R}^2\to\mathbb{R}^2$ the {\itshape orbit} of $x$ is given by $\mathcal{O}(x)=\{h^n(x):n\in\mathbb{Z}\}$. 
The present paper is concerned with the existence of fixed points of orientation preserving homeomorphisms of $\mathbb{R}^2$ linked to bounded orbits, in the sense introduced by J.-M. Gambaudo in  \cite{Gambaudo}. Let $\mathcal{O}_1$ and $\mathcal{O}_2$ be two sets invariant for a homeomorphism $h$. Following \cite{Gambaudo} we say that these two sets are {\it unlinked} if there exist two discs $D_1,D_2\subset \mathbb{R}^2$ with the following properties:
\begin{itemize}
    \item $\mathcal{O}_i\subset \operatorname{int}(D_i)$ for i=1,2;
    \item $D_1\cap D_2=\emptyset$;
    \item $h(\partial D_i)$ is isotopic to $\partial D_i$ in $\mathbb{R}^2\setminus\left(\mathcal{O}_1\cup\mathcal{O}_2\right)$, for $i=1,2$.
\end{itemize}
Note that if $\cl(\mathcal{O}_1)\cap\mathcal{O}_2\neq\emptyset$ then $\mathcal{O}_1$ and $\mathcal{O}_2$ are linked in a trivial way, as there is no Jordan curve around $\mathcal{O}_2$ separating the two sets. Gambaudo showed that for any $C^1$-embedding $f$ of a disk, and any periodic orbit $O$ of $f$, there exists a fixed point $p$ linked to $O$. In consequence, for the torus flow $\phi_t$ suspending $f$, the sets $\{\phi_t(O)\}_{\geq 0}$ and $\{\phi_t(p)\}_{\geq 0}$ are linked as knots in $\mathbb{S}^3$. A similar result was obtained simultanously by B. Kolev in \cite{Kolev}, who showed linking of a periodic orbit to a fixed point for orientation preserving $C^1$-diffeomorphisms of $\mathbb{R}^2$. The result of Gambaudo and Kolev was generalized to orientation reversing homeomorphisms of $\mathbb{S}^2$ by M. Bonino \cite{Bo2}, who showed linking of periodic orbits of period at least $3$, to periodic orbits of least period $2$. Bonino also pointed out that Kolev's proof, in the orientation preserving case, works in $C^0$ as well, as it is enough to perturb a given homeomorphism slightly, by smoothing it out in a small neighborhood of the periodic orbit and then apply the same proof. In the present paper we improve on the results of Gambaudo and Kolev, by proving that {\it any} bounded orbit is linked to a fixed point. 
\begin{theorem}\label{main1}
Let $h:\mathbb{R}^2 \to  \mathbb{R}^2$ be an orientation preserving homeomorphism. For any bounded orbit $\mathcal{O}(x)$ there exists a fixed point $p\in\mathbb{R}^2$ linked to $\mathcal{O}(x)$. 
\end{theorem}
{\bf Outline of Proof. }
\begin{itemize}{\it
\item  Suppose no point in $\operatorname{Fix}(h)$ is linked to the compact $K=\cl\mathcal{O}(x)$. 
\item Then there exists a lift $\tilde{h}$ of $h$ to the universal cover $(\tau,\tilde{U})$ of an open surface $U$ (which is the component of $\mathbb{R}^2\setminus\operatorname{Fix}(h)$ containing $K$), and a (compact) sheet $\tilde{K}$ over $K$, such that $\tilde{h}(\tilde K)=\tilde K$. This step is possible since $h(U)=U$, as a consequence of a result of Brown and Kister \cite{BM2}. 
\item Since $\tilde{h}$ is orientation preserving and fixed point free, and $\tilde K$ is compact, we obtain a contradiction with Brouwer Translation Theorem.
\item To show that $\tilde{K}$ is invariant under $\tilde{h}$, one projects $\tilde{K}$ from the universal cover onto an infinite cyclic covering space, to realize that $\tilde K$ is $\tilde{h}$-invariant unless a fixed point of $h$ linked to $K$ already exists.}
\end{itemize}

Our proof is inspired by M. Brown's paper \cite{BM}, the writing of \cite{Boronski}, and motivated by \cite{Kuperberg} where Brown's approach was used to show existence of periodic points in neighborhoods of adding machines in the plane. Unlike the proofs of results in \cite{Bo2}, \cite{Gambaudo} and \cite{Kolev}, our proof does not employ any elements of the Nielsen-Thurston theory, but purely topological arguments concerning covering spaces of open surfaces. We believe that this proof reveals a deeper explanation of the phenomenon described by the result of Gambaudo and Kolev.  
Our result is applicable not only to periodic orbits, but also to non-periodic bounded orbits, and can be used, for example, to guarantee fixed points linked to an invariant Cantor set $C=\cl(\mathcal{O}(x))$. This is particularly useful when $h|C$ is minimal\footnote{$h|C$ is said to be minimal if the orbit of every point is dense in $C$}, or more generally aperiodic\footnote{$h|C$ is aperiodic if $h|C$ has no periodic orbits}, as then $C$ does not contain any fixed points. It is easy to find examples of local dynamics of $h$ for which $h$ may be such that the result of Kolev and Gambaudo does not guarantee a linked fixed point, but the above theorem does. In particular this may occur for a disk $D\subseteq \mathbb{R}^2$ such that $h|\partial D$ is conjugate to a circle homeomorphism with irrational rotation number (minimal or Denjoy-type), or $h|C$ is an odometer with $C\subseteq D$. However, the set of minimal homeomorphisms of the Cantor set is very rich, and contains homeomorphisms with any prescribed topological entropy, even topologically mixing \cite{Lehrer}, and so potential applications go far beyond these two examples\footnote{Note that any Cantor set homeomorphism extends to a homeomorphism of $\mathbb{R}^2$ \cite[Chapter 13, Theorem 7, p. 93]{Moise}.}. Moreover, our result also applies to orbits whose closures might be connected sets, even separating plane continua. Finally, our theorem can also be used to detect linking of one-dimensional invariant sets of suspension flows in $\mathbb{S}^3$, more general than just knots considered in \cite{Gambaudo}, such as 1-dimensional {\it matchbox manifolds}; i.e. the class of compact connected metrizable spaces in which every point has a neighborhood homeomorphic to the product $[0,1]\times C$. This class of spaces includes many familiar examples from the dynamics literature, such as Denjoy exceptional minimal sets of flows \cite{Schweitzer}, solenoids \cite{Smale}, and DA-attractors \cite{Katok}; see e.g. \cite{CHL} and \cite{Fokkink} to learn more. 
\section{Proof of Theorem \ref{main1}}
Suppose $\mathcal{O}_1=\{h^n(x):n\in\mathbb{Z}\}$ is an orbit and $\mathcal{O}_2=\{p\}$ is a fixed point of $h$, such that $p\notin \cl(\mathcal{O}_1)$. If $D_2$ is a sufficiently small disk containing $\mathcal{O}_2$ then $\partial D_2$ is always isotopic to $h(\partial D_2)$\footnote{Note that this shows that this form of linking is not well suited for homeomorphisms of $\mathbb{S}^2$, since there any two Jordan curves separating $\mathcal{O}_1$ and $x'$ are isotopic in $\mathbb{S}^2\setminus(\mathcal{O}_1\cup\{x'\})$. This contrasts with the case discussed in \cite{Bo2}, when $\mathcal{O}_2$ is of least period $2$, in which case this form of linking is meaningful.}. Therefore for linking of orbits to fixed points, Gambaudo's definition reduces to the following condition. 
\begin{itemize}
    \item $\mathcal{O}_1$ and $p$ are {\it linked} if there does not exist a closed disk $D_1$ with $\mathcal{O}_1\subset\operatorname{int}(D_1)$, such that $p\notin D_1$ and $\partial D_1$ is isotopic to $h(\partial D_1)$ in $\mathbb{R}^2\setminus\left(\mathcal{O}_1\cup\{p\}\right)$.
\end{itemize}
If such a Jordan curve $\partial D_1$ does exist, and consequently $\mathcal{O}_1$ and $p$ are unlinked, then we shall call $\partial D_1$ an {\it unlinking} of  $\mathcal{O}_1$ and $p$. 

Recall that if $U$ is an open surface and $(\tilde{U},\tau)$ is its universal covering space then given a homeomorphism $h:U\to U$ there exists a {\itshape lift} homeomorphism $\tilde{h}:\tilde{U}\to \tilde{U}$ such that the following diagram commutes.
\[
\begin{CD}
\tilde{U}@>\tilde{h}>> \tilde{U} \\
@V\tau VV @V\tau VV \\
U @>h>>U
\end{CD}
\]
Additionally if $h(x)=y$ then $\tilde{h}$ is uniquely determined by the choice of two points $\tilde{x}\in\tau^{-1}(x)$,  $\tilde{y}\in\tau^{-1}(y)$ and setting $\tilde{h}(\tilde{x})=\tilde{y}$. If $D\subset U$ is a disk (or its subset), then $\tau^{-1}(D)$ consists of pairwise disjoint homeomorphic copies of $D$ in $\tilde{U}$, called {\it sheets}. We shall need the following celebrated result of Brouwer. Brouwer's theorem with its subsequent generalizations is much stronger, but we shall only need the weaker version stated below. The reader is referred to \cite{Slaminka} for a historical account of the proof of Brouwer's result. 
\begin{theorem}[Brouwer Translation Theorem]\cite{Br}
Let $h : \mathbb{R}^2 \to  \mathbb{R}^2$ be an orientation preserving planar homeomorphism. If there exists an $x\in\mathbb{R}^2$ such that $\mathcal{O}(x)$ is bounded, then there exists a fixed point $p\in\mathbb{R}^2$ of $h$. 
\end{theorem}
Now we are ready to prove Theorem \ref{main1}. 
\begin{proof}(of Theorem \ref{main1})
Let $\mathcal{O}(x)$ be a bounded orbit and $p\in\operatorname{Fix}(h)$. If $p\in\cl(\mathcal{O}(x))\cap \operatorname{Fix}(h)$ then $\mathcal{O}(x)$ and $p$ are linked, and so from now on we shall assume that $\cl(\mathcal{O}(x))\cap \operatorname{Fix}(h)=\emptyset$. 
\vspace{0.25cm}

\noindent
{\bf Standing Assumption:} $\cl(\mathcal{O}(x))\cap \operatorname{Fix}(h)=\emptyset$.
\vspace{0.25cm}

\noindent
We  start with the case when $K=\cl(\mathcal{O}(x))$ does not separate $\mathbb{R}^2$, as the case when $K$ separates is easier.
\vspace{0.25cm}

\noindent
{\bf CASE I:} $\cl(\mathcal{O}(x))$ does not separate $\mathbb{R}^2$.
\vspace{0.25cm}

\noindent
By contradiction, suppose that no point in $\operatorname{Fix}(h)$ is linked to $\mathcal{O}(x)$. Let $U$ be the component of $\mathbb{R}^2\setminus \operatorname{Fix}(h)$ that contains $x$. Recall that $h(U)=U$ by \cite{BM2}, and so $K\subset U$. Note that $U$ cannot be simply connected, as otherwise we obtain a contradiction with Brouwer Translation Theorem, since $U$ is then homeomorphic to $\mathbb{R}^2$ and contains a bounded orbit, but no fixed point.
Let $x_n=h^n(x)$ for each $\mathbb{N}$. Let $(\tilde{U},\tau)$ be the universal cover of $U$. Note that $\tilde{U}$ is homeomorphic to $\mathbb{R}^2$ and, since $K$ is contained in a disk disjoint from $\operatorname{Fix}(h)$, $K$ lifts to pairwise disjoint homeomorphic copies of itself (sheets) in $\tilde U$, each of which maps homeomorphically onto $K$ by $\tau$. Let $\tilde K$ be one such a sheet. We have $\tau(\tilde{K})=K$. Let $\tilde{x}=\tau^{-1}(x)\cap \tilde{K}$ and $\tilde{x}_1=\tau^{-1}(x_1)\cap \tilde{K}$. The homeomorphism $h$ lifts to a unique homeomorphism $\tilde h$ such that $\tilde{h}(\tilde{x})=\tilde x_1$. 
\begin{claim}\label{h2}
$\tilde{h}^n(\tilde{x})\in\tilde{K}$ for every $n\in\mathbb{N}$.
\end{claim}
\begin{proof}(of Claim \ref{h2})
By contradiction, suppose that there exists an $N$ such that $\tilde{h}^N(\tilde{x})\notin\tilde{K}$. 
\begin{figure}[h]
		\centering		\includegraphics[width=12cm,height=12cm]{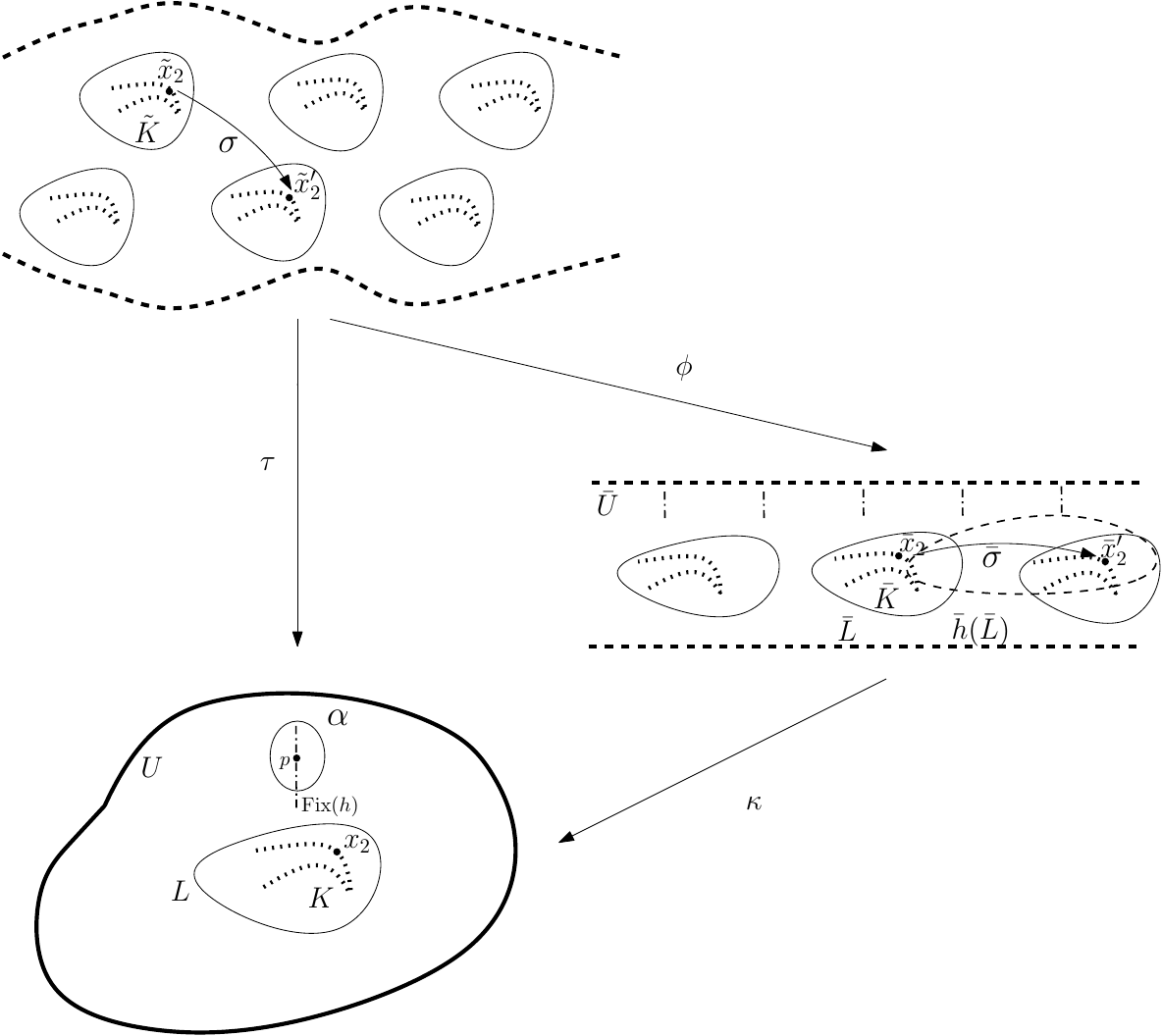}
		\caption{Proof of Theorem \ref{main1}, CASE I: The component $U$ of $\mathbb{R}^2\setminus \operatorname{Fix}(h)$ containing $\cl(\mathcal{O}(x))$ and the covering spaces $\tilde{U}$ and $\bar U$. }
		\label{fig:cover}
	\end{figure}
Without loss of generality we assume that $N=2$. Let $\tilde{x}_2'=\tau^{-1}(x_2)\cap \tilde K$ and $\tilde x_2=\tilde h(\tilde x_1)$. We have $\tilde x_2\in \tilde h(\tilde K)\neq \tilde K$.

Let $\sigma:\tilde{U}\to\tilde{U}$ be the deck transformation such that $\sigma(\tilde x_2')=\tilde x_2$. The deck transformation group is isomorphic to the fundamental group $\pi_1(U)$ of $U$, and one sees $\sigma$ as an element $\alpha$ of $\pi_1(U)$. There exists a point $p\in\operatorname{Fix}(h)$, such that $\alpha$ is a nontrivial loop in the surface $W=\mathbb{R}^2\setminus \{p\}$. Let $(\bar W, \kappa)$ be an infinite cyclic covering space of $W$. Then $(\tilde{U},\tau)$ is also a universal covering space of $\bar U$, the component of $\bar W\setminus \kappa^{-1}(\operatorname{Fix}(h)\setminus\{p\})$ that contains $\kappa^{-1}(K)$, and so there exists a covering map $\phi:\tilde U\to\bar U$ such that $\tau=\kappa|_{\bar U}\circ\phi$. Let $\bar x=\phi(\tilde x), \bar x_1=\phi(\tilde x_1),\bar x_2=\phi(\tilde x_2)$ and $\bar x'_2=\phi(\tilde x'_2)$. By the choice of $\alpha$ we have that the element of the deck transformation group (with respect to $\bar U$) $\bar\sigma$ satisfying $\bar \sigma\circ\phi=\phi \circ \sigma$ is nontrivial, and $\bar \sigma (\bar x_2')=\bar x_2$, so $\bar x_2\neq \bar x_2'$. Let $L$ be an unlinking of $\cl(\mathcal{O}(x))$ and $\{p\}$. Set $\bar K=\phi(\tilde{K})$, and let $\bar L$ be a sheet over $L$ that bounds a disk containing
$\bar K$, and 
$\bar h$ be the lift of $h$ to $\bar W$ given by $\bar h(\bar x)=\bar x_1$. We have $\phi\circ \tilde h=(\bar h|\bar U)\circ\phi$, and $\bar h(\bar x_1)=\bar x_2$. 

Consider an isotopy $\{i_t:\mathbb{S}^1\rightarrow \mathbb{R}^2\setminus\left(\cl(\mathcal{O}(x))\cup\{p\}\right):0\leq t\leq 1\}$ from $i_0(\mathbb{S}^1)=L$ to $i_1(\mathbb{S}^1)=h(L)$. By isotopy lifting property, this isotopy lifts to an isotopy $\tilde{i}_t:\mathbb{S}^1\to\bar W\setminus\kappa^{-1}(\cl(\mathcal{O}(x)))$ taking $\bar{L}$ to $\bar{h}(\bar{L})$, both of which are loops in $
\bar W$ (since $L$ and $h(L)$ are inessential in $\mathbb{R}^2\setminus\{p\}$). This leads to a contradiction, since $\bar{L}$ bounds a disk containing $\bar{x}_2$, but $\tilde{h}(\bar{L})$ does not, so $\tilde{i}_t$ cannot take $\bar{L}$ to $\bar{h}(\bar{L})$ in $\bar W\setminus\kappa^{-1}(\cl(\mathcal{O}(x)))$. This completes the proof of Claim \ref{h2}.
\vspace{0.25cm}
\end{proof}
To conclude the proof of CASE I it is now enough to observe that $\tilde{h}$ is an orientation preserving homeomorphism of the plane $\tilde{U}$, with a compact invariant set $\tilde K$, but no fixed points, contradicting Brouwer Translation Theorem.

\vspace{0.25cm}
\noindent
{\bf CASE II: $\cl(\mathcal{O}(x))$ separates $\mathbb{R}^2$.}
\vspace{0.25cm}

\noindent
If no bounded component $V$ of $\mathbb{R}^2\setminus\cl(\mathcal{O}(x))$ contains a fixed point then we add all such components to $\cl(\mathcal{O}(x))$ and we are back to CASE I. 
\vspace{0.25cm}

\noindent
Otherwise, there exists a bounded component $V_o$ of $\mathbb{R}^2\setminus\cl(\mathcal{O}(x))$ such that $h(V_o)= V_o$ and $V_o$ contains a fixed point $p'$ of $h$. But since $\partial V_o$ separates the plane into at least two components, one of which contains $p'$, and none of which contains $\cl(\mathcal{O}(x))$, there does not exist a disk $D_1$ such that $\cl(\mathcal{O}(x))\subset \operatorname{int}(D_1)$ and $p'\notin D_1$. Consequently there is no unlinking of $\cl(\mathcal{O}(x))$ and $p'$, and these two orbits must be linked. This concludes the proof of CASE II. 

The proof of Theorem \ref{main1} is complete. 
\end{proof}
\begin{remark}
The proof of Theorem \ref{main1} remains valid if the full orbit $\mathcal{O}(x)$ is replaced with one of the half orbits $\mathcal{O}^+(x)=\{h^n(x):n\in\mathbb{N}\}$ or $\mathcal{O}^-(x)=\{h^{-n}(x):n\in\mathbb{N}\}$.
\end{remark}
\section{Final Remarks}
In 1988 J. Franks defined what seems to be a deeper form of linking, for which one requires from a periodic orbit to have a non-zero rotation number around a fixed point. Franks asked whether every periodic orbit of an orientation preserving homeomorphism is linked in that sense to a fixed point \cite{FB}. This difficult open problem was resolved in the affirmative by P. Le Calvez in 2006 \cite{LeCalvez2}. It seems that Le Calvez's result cannot be extended to bounded orbits. A result that seems related to both forms of linking is proven in \cite{Graff}, where a way of locating fixed points in proximity of recurrent orbits is given, by the means of topological hulls of unions of arcs, connecting a finite number of points of an $\epsilon$-periodic orbit; see also \cite{Kucharski} and \cite{Ostrovski}. Sufficient conditions for the nonremovability of collections of periodic points under isotopy relative to a general compact invariant set can be found in \cite{BH}.
\section{Acknowledgments}
I am grateful to Toby Hall for comments that helped to improve the paper, and an anonymous referee who pointed out a necessity to revise the initial proof of the main result. This work was supported in part by the National Science Centre, Poland (NCN), grant no. 2019/34/E/ST1/00237.

\bibliographystyle{alpha}

\end{document}